\documentclass[12pt,oneside]{amsart}

\usepackage{amssymb}
\usepackage{amscd}

\title{Rational curves of degree four with two inner Galois points} 
\author{Satoru Fukasawa}

\subjclass[2010]{14H50}
\keywords{Galois point, plane curve, rational curve, degree four}
\address{Department of Mathematical Sciences, Faculty of Science, Yamagata University, Kojirakawa-machi 1-4-12, Yamagata 990-8560, Japan} 
\email{s.fukasawa@sci.kj.yamagata-u.ac.jp} 
\thanks{The author was partially supported by JSPS KAKENHI Grant Number 25800002.}

\newtheorem{proposition}{Proposition}
 
\newtheorem{conjecture}{Conjecture} 
\newtheorem{fact}{Fact}

\newtheorem{remark}{Remark}
\newtheorem{lemma}{Lemma}

\begin{document}
\begin{abstract} 
We characterize plane rational curves of degree four with two or more inner Galois points. 
A computer verifies the existence of plane rational curves of degree four with three inner Galois points. 
This would be the first example of a curve with exactly three them. 
Our result implies that Miura's bound is sharp for rational curves. 
\end{abstract}

\maketitle
\section{Introduction}  
Let the base field $K$ be an algebraically closed field of characteristic $p \ge 0$ and let $C \subset \Bbb P^2$ be an irreducible plane curve of degree $d \ge 4$. 
For a point $P \in C$, if the function field extension $K(C)/\pi_P^{*}K(\Bbb P^1)$ induced by the projection $\pi_P$ from $P$ is Galois, then the point $P$ is said to be (inner) {\it Galois} for $C$. 
This notion was introduced by Hisao Yoshihara (\cite{miura-yoshihara, yoshihara}). 
We denote by $\delta(C)$ the number of Galois points which are smooth points of $C$. 
It is a natural problem: {\it What is $\delta(C)$?} 
If $p=0$ and $C$ is smooth, Yoshihara and Miura determined $\delta(C)$ completely (\cite{miura-yoshihara, yoshihara}). 
However, this problem would be difficult in general. 
For example, only six (only two! if $p=0$) types  of curves with $\delta(C) \ge 2$ are known (see the table in \cite{yoshihara-fukasawa}). 

Throughout this note, we assume that $p \ne 2, 3$, $d=4$ and $C$ is rational. 
Let $\varphi: \Bbb P^1 \rightarrow C$ be the normalization. 
Note that the morphism $\pi_P \circ \varphi$, where $P \in C$ is smooth and Galois, has only tame ramification under the assumption that $p \ne 3$ and $d=4$ (\cite[III. 7.2]{stichtenoth}). 
We use the same notation as in \cite{miura}. 
We denote by $\delta_2(C)$ (resp. $\delta_0(C)$) the number of Galois points $P \in C$ which are smooth points with $I_P(C, T_PC)=4$ (resp. $I_P(C, T_PC)=2$), where $I_P(C, T_PC)$ is the intersection multiplicity of $C$ and the tangent line $T_PC$ at $P$. 
As noted in \cite{miura}, $\delta(C)=\delta_2(C)+\delta_0(C)$. 
In this note, we characterize curves with $\delta(C) \ge 2$. 

\begin{proposition} \label{TotalFlexes}
Let $C \subset \Bbb P^2$ be a plane rational curve of degree four. 
We have $\delta_2(C) \ge 2$ if and only if $C$ is projectively equivalent to the curve which is the image of the morphism 
$$ \Bbb P^1 \rightarrow \Bbb P^2; (s:t) \mapsto (s^4:s(s+t)^3:t^4).  $$
In this case, $\delta(C)=\delta_2(C)=2$. 
\end{proposition}

\begin{remark}
The curve described in Proposition \ref{TotalFlexes} has exactly three ordinary double points, since $\varphi(1, t)=\varphi(1, u)$ for three pairs 
\begin{eqnarray*} 
(t, u)&=&(i\sqrt{3},\ -i\sqrt{3}), \\ 
& & \left(\frac{1}{2}\left(-3+\sqrt{3}+i\sqrt{12-6\sqrt{3}}\right),\ \frac{1}{2}\left(-3+\sqrt{3}-i\sqrt{12-6\sqrt{3}}\right)\right), \\
& & \left(\frac{1}{2}\left(-3-\sqrt{3}+i\sqrt{12+6\sqrt{3}}\right),\ \frac{1}{2}\left(-3-\sqrt{3}-i\sqrt{12+6\sqrt{3}}\right) \right). 
\end{eqnarray*} 
Therefore, this is of type IIIg in the table of \cite[Theorem 2]{miura}. 
\end{remark}

\begin{proposition} \label{Total2}
There exists no plane rational curve of degree four with $\delta_0(C) \ge 1$ and $\delta_2(C) \ge 1$. 
\end{proposition}

\begin{proposition} \label{NonFlexes} 
Let $C \subset \Bbb P^2$ be a plane rational curve of degree four. 
We have $\delta_0(C) \ge 2$ if and only if $C$ is projectively equivalent to the curve which is the image of the morphism 
$$ \Bbb P^1 \rightarrow \Bbb P^2; (s:t) \mapsto (s(t+\alpha s)^3:t(t+\beta s)^3:t(t+s)^3),  \hspace{10mm} (*)$$
where $\beta(\beta-1)(\beta^2+\beta+1) \ne 0$ and $\alpha=(\beta^2+\beta+1)/3$.  

Furthermore, if the differential map of $\varphi$ between Zariski tangent spaces at some points is zero, then $C$ is projectively equivalent to the curve 
$$ \Bbb P^1 \rightarrow \Bbb P^2; (s:t) \mapsto (s^4:ts^3:t(t+s)^3). $$
In this case, $\delta(C)=\delta_0(C)=2$. 
\end{proposition}

\begin{remark}
The curve given by
$$ \Bbb P^1 \rightarrow \Bbb P^2; (s:t) \mapsto (s^4:ts^3:t(t+s)^3) $$
is projectively equivalent to the curve in \cite[Example 1]{miura}. 
Therefore, our family contains Miura's example. 
\end{remark}

A computer (the author used MATHEMATICA \cite{wolfram}) verifies the following conjecture. 
 (For singularities, we prove only when $p=0$.)

\begin{conjecture} \label{main} 
We have $\delta_0(C)=3$ if and only if $C$ is projectively equivalent to the curve given by $(*)$ with
$$ \beta^4+3\beta^3+\beta^2+3\beta+1=0.$$ 
In this case, the curve has three ordinary double points, i.e. it belongs to the class $IIIg$ in the table of \cite[Theorem 2]{miura}. 
\end{conjecture}

Our result implies that the following bound for rational curves due to Miura \cite[Theorem 1]{miura} is sharp. 

\begin{fact}[Miura] 
If $p=0$, $C \subset \Bbb P^2$ is rational and $d-1$ is prime, then $\delta(C) \le 3$. Furthermore, $\delta(C)=3$ only if $d=4$ and $C$ has no cusp. 
\end{fact} 

In \cite[p. 291]{miura}, Miura says ``For the remaining cases, we can not obtain a better estimation than Theorem 1. ... Problems. (1) {\it Find a better estimation $\delta(C)$ for the case when $C$ has no cusp.}''  
Our result gives an answer to this problem.

\begin{remark}
If $Q \in {\rm Sing}(C)$ and $\varphi^{-1}(Q)=\{Q_1, Q_2, Q_3\}$, then $\delta(C) \le 1$. 

Assume by contradiction that $P_1, P_2$ are inner Galois points. 
Since $Q_1, Q_2, Q_3$ are contained in the same fiber for $\pi_{P_i}\circ \varphi$ for $i=1, 2$, there exists $\sigma_i \in G_{P_i}$ of order three such that $\sigma_i(Q_1)=Q_2$ (\cite[III. 7.1]{stichtenoth}). 
Then, we have $\sigma_i(Q_2)=Q_3$ and $\sigma_i(Q_3)=Q_1$. 
Since $\sigma_1, \sigma_2$ are automorphisms of $\Bbb P^1$, $\sigma_1=\sigma_2$. 
Then, the sets of ramification points are the same for $\pi_{P_1}$ and $\pi_{P_2}$.  
This is a contradiction. 

By this, $\delta(C) \le 1$ for the type II1/2a in the table of \cite[Theorem 2]{miura}. 
\end{remark}

\section{Proof of Propositions \ref{TotalFlexes} and \ref{Total2}} 
First we note the following. 

\begin{lemma} \label{Galois} 
Let $e$ be a prime number different from $p$ and let $\theta:\Bbb P^1 \rightarrow \Bbb P^1$ be a surjective morphism of degree $e$. 
The morphism $\theta$ is a (cyclic) Galois covering if and only if there exist two different points $P_1, P_2$ such that the ramification indices at $P_1, P_2$ are equal to $e$.  
\end{lemma}

\begin{proof} 
The only-if part is well-known (see \cite[III. 7.2]{stichtenoth}). 
We prove the if-part. 
Let $\theta=(f(s,t):g(s,t))$, where $f, g$ are homogeneous polynomials of degree $e$. 
We may assume that $P_1=(0:1)$, $P_2=(1:0)$, $\theta(P_1)=(0:1)$ and $\theta(P_2)=(1:0)$. 
Since $\theta$ is ramified at $P_1$ with index $e$, we have $f(s, t)=a s^e$ for some $a \in K \setminus 0$. 
We also have $g(s, t)=b t^e$ for some $b \in K \setminus 0$. 
Therefore, we may assume that $\theta(s:t)=(s^e:t^e)$. 
This implies that the morphism $\theta$ is a cyclic Galois covering. 
\end{proof}

We assume that $\delta_2(C) \ge 1$ and $\delta(C) \ge 2$. 
Let $P_1, P_2 \in C$ be smooth points such that $P_1, P_2$ are Galois, and let $Q_1, Q_2$ be (the images of) ramification points different from $P_1, P_2$ for $\pi_{P_1}, \pi_{P_2}$ respectively. 
Assume that $I_{P_1}(C, T_{P_1}C)=4$.  
Note that the lines $T_{P_1}C$, $\overline{P_1Q_1}$, $\overline{P_2Q_2}$ are not concurrent, where $\overline{P_1Q_1}$ is the line passing through $P_1, Q_1$. 
Therefore, we take a suitable system of coordinates so that they are defined by $X=0$, $Y=0$, $Z=0$ respectively. 
Let $P_1=\varphi(0:1)$, $Q_1=\varphi(1:-1)$, $P_2=\varphi(1:0)$ and let $Q_2=\varphi(1:-\alpha)$, where $\alpha \ne 0$. 
Considering the intersection numbers by lines $X=0$, $Y=0$ $Z=0$, the morphism $\varphi$ is represented by 
$$ \Bbb P^1 \rightarrow \Bbb P^2; (s:t) \mapsto (s^4:s(s+t)^3:t(t+\alpha s)^3). $$
We consider the projection $\pi_{P_2}$. 
Since $P_2=\varphi(1:0)=(1:1:0)$, $\pi_{P_2}$ is represented by $(Y-X:Z)$. 
Then, 
$$\pi_{P_2}\circ\varphi=(t^2+3t+3:(t+\alpha)^3) $$
and 
\begin{eqnarray*}
& & \det\left(\begin{array}{cc} \pi_{P_2}\circ \varphi \\ d_t(\pi_{P_2}\circ \varphi) 
\end{array}\right) 
=\det\left(\begin{array}{cc} t^2+3t+3 & (t+\alpha)^3  \\
2t+3 & 3(t+\alpha)^2 
\end{array}\right) \\
&=&(t+\alpha)^2(t^2-2(\alpha-3)t-3(\alpha-3)). 
\end{eqnarray*} 
The discriminant of $t^2-2(\alpha-3)t-3(\alpha-3)$ is $4\alpha(\alpha-3)$. 
Using Lemma \ref{Galois}, we have that the point $P_2$ is Galois if and only if $\alpha=3$. 
If $\alpha=3$, $\varphi=(1:(t+1)^3:t(t+3)^3)$. 
Then, we have $I_{P_2}(C, T_{P_2}C)=4$. 
Let $M(C)$ be the integer such that $I_R(C, T_RC)=M(C)$ for a general point $R \in C$. 
Then, $M(C)=2$ and 
$$ \sum_{R \in C \setminus {\rm Sing}(C)} (I_R(C, T_RC)-2) \le (2+1)(2\times 0-2)+3 \times 4=6$$
(\cite[Theorem 1.5]{stohr-voloch}). 
By this inequality, points $P_1, P_2, Q_1, Q_2$ are all flexes.  
Therefore, we have $\delta(C)=\delta_2(C)=2$. 

Note that $9-9(t+1)^3+t(t+3)^3=t^4$. 
This implies that the curve is projectively equivalent to 
$$ \Bbb P^1 \rightarrow \Bbb P^2; (s:t) \mapsto (s^4:s(s+t)^3:t^4). $$
 
\begin{remark}
The if-part of Proposition \ref{TotalFlexes} holds true also in $p=2$. 
If $p=2$, then the described curve is projectively equivalent to the curve given by
$$ \Bbb P^1 \rightarrow \Bbb P^2; (s:t) \mapsto (s^4:st^3:t^4). $$
According to a result of Fukasawa and Hasegawa \cite[Example 1]{fukasawa-hasegawa}, we have $\delta(C)=\infty$. 
\end{remark}

\section{Proof of Proposition \ref{NonFlexes}} 
We assume that $\delta_0(C) \ge 2$.  
Let $P_1, P_2 \in C$ be smooth points such that $P_1, P_2$ are Galois with $I_{P_1}(C, T_{P_1}C)=I_{P_2}(C, T_{P_2}C)=2$, and let $Q_1, Q_2$ be (the images of the) ramification points different from $P_1, P_2$ for $\pi_{P_1}, \pi_{P_2}$ respectively. 
Further, we assume that $R_2$ is a ramification point $\ne P_2, Q_2$ for $\pi_{P_2}$. 
Note that the lines $\overline{P_1Q_1}$, $\overline{P_2Q_2}$, $\overline{P_2R_2}$ are not concurrent. 
Therefore, we take a suitable system of coordinates so that they are defined by $X=0$, $Y=0$, $Z=0$ respectively. 
Let $P_1=\varphi(0:1)$, $P_2=\varphi(1:0)$, $Q_1=(1:-\alpha)$, $Q_2=\varphi(1:-\beta)$ and let $R_2=(1:-\gamma)$, where $\alpha, \beta, \gamma \in K \setminus 0$ and $\beta \ne \gamma$. 
Considering the intersection numbers by lines $X=0$, $Y=0$ $Z=0$, the morphism $\varphi$ is represented by 
$$ \Bbb P^1 \rightarrow \Bbb P^2; (s:t) \mapsto (s(t+\alpha s)^3:t(t+\beta s)^3:t(t+\gamma s)^3). $$
We consider the projection $\pi_{P_1}$. 
Since $P_1=\varphi(0:1)=(0:1:1)$, $\pi_{P_1}$ is represented by $(X:Y-Z)$. 
Then, 
$$\pi_{P_1}\circ\varphi=((t+\alpha)^3:t(t+\beta)^3-t(t+\gamma)^3). $$
Let $g(t)=t(t+\beta)^3-t(t+\gamma)^3=(\beta-\gamma)t(3t^2+3(\beta+\gamma)t+(\beta^2+\beta\gamma+\gamma^2))$. 
Note that the rational map $\pi_{P_1} \circ \varphi$ has degree less than three if and only if $3\alpha^2-3(\beta+\gamma)\alpha+(\beta^2+\beta\gamma+\gamma^2)=0$. 
We have 
$$\det\left(\begin{array}{cc} \pi_{P_1}\circ \varphi \\ d_t(\pi_{P_1}\circ \varphi) 
\end{array}\right) 
=\det\left(\begin{array}{cc} (t+\alpha)^3 & g  \\
3(t+\alpha)^2 & g'
\end{array}\right) 
=(t+\alpha)^2((t+\alpha)g'-3g), 
$$
where
$$\frac{1}{\beta-\gamma}((t+\alpha)g'-3g)=(9\alpha-3(\beta+\gamma))t^2+(6(\beta+\gamma)\alpha-2(\beta^2+\beta\gamma+\gamma^2))t+(\beta^2+\beta\gamma+\gamma^2)\alpha $$
The discriminant is
$$ 4\{9\beta\gamma\alpha^2-3(\beta+\gamma)(\beta^2+\beta\gamma+\gamma^2)\alpha+(\beta^2+\beta\gamma+\gamma^2)^2\}. $$
Using Lemma \ref{Galois}, we have that the point $P_1$ is Galois if and only if  
$$\alpha=\frac{\beta^2+\beta \gamma+\gamma^2}{3\beta} \ \mbox{ or } \ \alpha=\frac{\beta^2+\beta \gamma+\gamma^2}{3\gamma}. $$

If the differential of $\varphi$ is zero at some point, then $\alpha=\beta$ or $\alpha=\gamma$. 
Then, $\varphi$ is represented by 
$$ \Bbb P^1 \rightarrow \Bbb P^2; (s:t) \mapsto (s^4:st^3:t(t+s)^3). $$

By taking a suitable system of coordinates, we may assume that $\alpha=(\beta^2+\beta\gamma+\gamma^2)/3\gamma$ and $\gamma=1$. 
Since $\alpha \ne 0$, we have $\beta^2+\beta+1 \ne 0$. 
The birationality of $\varphi$ is derived from the condition $3\alpha^2-3(\beta+1)\alpha+(\beta^2+\beta+1) \ne 0$. 

\begin{remark}
The if-part of Proposition \ref{NonFlexes} holds true also in $p=2$. 
If $p=2$, then $\alpha=\beta^2+\beta+1$ and the morphism $\pi_{P_1} \circ \varphi$ is a Galois covering, since 
$$ (t+\alpha)^3+\alpha g(t)=(\beta t+\alpha)^3. $$
\end{remark} 

\section{Computer-aided proof of Conjecture \ref{main}} 
We give the computer-aided proof of Conjecture \ref{main}. 
Let $C$ be the plane curve given by $(*)$ as in Proposition \ref{NonFlexes}. 
\begin{itemize}
\item[(1)] We compute the Hessian matrix $H$ and its determinant: 
$$ H= 
\left(\begin{array}{ccc} (t+\alpha)^3 & t(t+\beta)^3 & t(t+1)^3 \\
((t+\alpha)^3)_t & (t(t+\beta)^3)_t & (t(t+1)^3)_t \\
((t+\alpha)^3)_{tt} & (t(t+\beta)^3)_{tt} & (t(t+1)^3)_{tt}
\end{array}
\right).
$$
We have two new flexes given by 
$$ \beta^3+\beta^2+\beta+2\beta^3t+\beta^2t+\beta t+2t+3\beta t^2=0, $$
other than $-1, -\alpha, -\beta, -\frac{\beta^2+\beta+1}{3\beta}$. 
Let $t_1, t_2$ be the solutions of the equation. 
\item[(2)] 
We compute the determinant 
$$ T(x, y, z, t)= 
\det \left(\begin{array}{ccc} x & y & z \\ (t+\alpha)^3 & t(t+\beta)^3 & t(t+1)^3 \\
((t+\alpha)^3)_t & (t(t+\beta)^3)_t & (t(t+1)^3)_t \\
\end{array}
\right), $$
since the linear homogeneous polynomial $T_1=T(x, y, z, t_1)$ (resp. $T_2=T(x, y, z, t_2)$) gives the tangent line at $t=t_1$ (resp. $t=t_2$).  
\item[(3)] 
We compute the pull-backs $\varphi^*{T_1}=T(\varphi(t), t_1)$ and $\varphi^*{T_2}=T(\varphi(t), t_2)$. 
\item[(4)]
We solve the equation $\varphi^*{T_1}=\varphi^*{T_2}=0$ on $t$, under the assumption that 
$$ \beta(\beta^3-1) \ne 0, \det H \ne 0. $$ 
\item[(5)] Computer gives the following solutions: 
$$ \beta^4+3\beta^3+\beta^2+3\beta+1=0 \mbox{ and } t=\frac{\beta^3+2\beta^2-\beta+1}{3}, $$
or $\beta=-2, -\frac{1}{2}$.
We can exclude the case $\beta=-2, -\frac{1}{2}$ by the condition $\alpha \ne 1, \frac{\beta^2+\beta+1}{3\beta} \ne \beta$. 
\end{itemize}

Next, we consider singularities when $p=0$. 
A computer gives $24$ solutions $u$:  
\begin{eqnarray*}
0&=& 1+u+12u^2+31u^3+61u^4+33u^5+36u^6-27u^7+9u^8 \\
0&=&1-8u+27u^2+40u^3+61u^4+48u^5+81u^6+18u^7+9u^8 \\
0&=&1+10u+39u^2-107u^3+16u^4+213u^5+195u^6+63u^7+9u^8
\end{eqnarray*}
for equations
$$ \frac{(t+\alpha)^3}{t(t+1)^3}=\frac{(u+\alpha)^3}{u(u+1)^3},\ \frac{(t+\beta)^3}{(t+1)^3}=\frac{(u+\beta)^3}{(u+1)^3},\ \beta^4+3\beta^3+\beta^2+3\beta+1=0  $$
under the assumption $t \ne u$. 
This shows that the curve has three ordinary double points (for each $\beta$).

\section{Confirmation of the if-part of Conjecture \ref{main}} 
Let $$\beta^4+3\beta^3+\beta^2+3\beta+1=0 \mbox{ and } u=\frac{\beta^3+2\beta^2-\beta+1}{3}. $$
We can confirm that $\varphi(1: u)$ is Galois, by a computer, as follows. 
\begin{itemize}
\item[(1)] The equation $$\beta^3+\beta^2+\beta+2\beta^3t+\beta^2t+\beta t+2t+3\beta t^2=0$$
has two solutions
$$ t_1=-\frac{1-\beta+2\beta^2+\beta^3}{3\beta}, t_2=-\frac{1+2\beta-\beta^2+\beta^3}{3\beta}. $$
\item[(2)] We compute the point $P=(x, y, z)=\varphi(1,u)$, the projection 
$$\pi_P\circ\varphi=(z(t+\alpha)^3-x t(t+1)^3, z t(t+\beta)^3-y t(t+1)^3),$$
and the determinant of the matrix 
$D=\left(\begin{array}{cc} \pi_P\circ \varphi \\ d_t(\pi_P\circ \varphi) 
\end{array}\right).$
\item[(3)] The polynomial $\det D$ has a factor $g(t)$ of degree four on $t$ with $g(u) \ne 0$. 
\item[(4)] We have $g(t_1)=g(t_2)=0$. 
This implies that $\pi_P\circ\varphi$ is ramified at $t=t_1, t_2$. 
\item[(5)] Since points $\varphi(1:t_1), \varphi(1:t_2)$ are flexes, by Lemma \ref{Galois}, $\varphi(1:u)$ is Galois. 
\end{itemize}

\

\begin{center} {\bf Acknowledgements} \end{center} 
The author is grateful to Professor Tsuyoshi Miezaki for helpful advice.

\end{document}